\documentclass{amsart}
\usepackage{amscd}
\usepackage{amssymb}

\newcommand{\mathdot}{{\mathbf{\scriptscriptstyle\bullet}}}
\def\cdh{\mathrm{cdh}}

\def\oo{\otimes}
\def\oozar{\otimes_\zar}

\def\Pic{\operatorname{Pic}}
\def\Psh{\operatorname{Psh}}
\def\Sh{\operatorname{Sh}}

\def\Sch{\operatorname{Sch}}
\def\Spec{\operatorname{Spec}}
\def\Sym{\operatorname{Sym}}

\def\zar{\mathrm{zar}}
\def\lra{\longrightarrow}
\def\map#1{\ {\buildrel #1 \over \lra}\ }
\def\smap#1{\ {\buildrel #1 \over \to}\ }
\newcommand{\comment}[1]{}	

\newcommand{\bbL}{\mathbb L}

\newcommand{\Ab}{\mathbf{Ab}}
\newcommand{\Ch}{\mathbf{Ch}}

\def\F{F}  
\def\cF{\mathcal F}
\def\cG{\mathcal G}

\newcommand\calL{\mathcal L}
\def\cO{\mathcal O}

\newcommand{\A}{\mathbb{A}}

\newcommand{\bbH}{\mathbb H}

\newcommand{\Q}{\mathbb{Q}}

\input xypic
\xyoption{all} 
\numberwithin{equation}{section}

\theoremstyle{plain}
\newtheorem{thm}[equation]{Theorem}
\newtheorem*{thm*}{Theorem}
\newtheorem{cor}[equation]{Corollary}
\newtheorem{lem}[equation]{Lemma}
\newtheorem{prop}[equation]{Proposition}

\newtheorem{variant}[equation]{Variant}

\theoremstyle{definition}
\newtheorem{defn}[equation]{Definition}

\theoremstyle{remark}
\newtheorem{rem}[equation]{Remark}
\newtheorem{ex}[equation]{Example}

\newtheorem{subrem}{Remark}[equation] 
\newtheorem*{notation}{Notation}

\begin{document}
\bibliographystyle{plain}

\title{$K$-theory of line bundles and smooth varieties}

\author{C. Haesemeyer}
\thanks{Haesemeyer was supported by ARC DP-170102328}
\address{School of Mathematics and Statistics, University of Melbourne,
VIC 3010, Australia}
\email{christian.haesemeyer@unimelb.edu.au}

\author{C. Weibel}
\thanks{Weibel was supported by NSF grant DMS-146502}
\address{Dept.\ of Mathematics, Rutgers University, New Brunswick,
NJ 08901, USA} \email{weibel@math.rutgers.edu}

\date{\today}

\begin{abstract}
We give a $K$-theoretic criterion for a quasi-projective variety
to be smooth. If $\bbL$ is a line bundle corresponding to an ample
invertible sheaf on $X$, it suffices that $K_q(X)\cong K_q(\bbL)$
for all $q\le\dim(X)+1$.
\end{abstract}

\maketitle

\section*{}

Let $X$ be a quasi-projective variety over a field $k$ of characteristic~0.
The main result of this paper gives a $K$-theoretic criterion for
$X$ to be smooth. For affine $X$, such a criterion was given in
\cite{chw-v}: it suffices that $X$ be $K_{d+1}$-regular for $d=\dim(X)$,
i.e., that $K_{d+1}(X)\cong K_{d+1}(X\times\A^m)$ for all $m$.
If $X$ is affine, we also showed that $K_{d+1}$-regularity of $X$
is equivalent to the condition that $K_i(X)\cong K_i(X\times\A^1)$
for all $i\le d+1$.

We also showed that $K_{d+1}$-regularity is insufficient for
quasi-projective $X$; see \cite[Thm.\,0.2]{chw-v}. In this paper we prove:

\begin{thm}\label{main}
Let $X$ be quasi-projective over a field $k$ of characteristic~0,
of dimension $d$, and let $\bbL=\Spec(\Sym\,\calL)$ be the line bundle
corresponding to an ample invertible sheaf $\calL$ on $X$.

If $K_i(\bbL)\cong K_{i}(X)$ for all $i\le n$
then $X$ is regular in codimension $<n$.

If $K_i(\bbL)\cong K_{i}(X)$ for all $i\le d+1$,
then $X$ is regular.
\end{thm}

For example, if $K_i(\bbL)\cong K_{i}(X)$ for all $i\le d$,
then $X$ has at most isolated singularities.

In the affine case, of course, every line bundle is ample, and
when $\bbL=\A^1_R$ we recover our previous result, proven in
\cite[0.1]{chw-v}:

\begin{cor}\label{Vorst}
If $R$ is essentially of finite type over a field of characteristic~0,
and $K_i(R)\cong K_i(R[t])$ for all $i\le n$ then
$R$ is regular in codimension $<n$.
\end{cor}

The affine assumption in this corollary
is critical.  In \cite{chw-v},
we gave an example of a curve $Y$ which is
$K_n$-regular for all $n$, but which is not regular;
no affine open $U$ is even reduced.
However, $K_0(X)\ne K_0(\bbL)$ for the line bundle associated to
an ample $\calL$; see Example \ref{ex:DD} below.
In Theorem \ref{thm:cusp-bundle} we give a surface $X$
which is $K_n$-regular for all $n$, but which is not regular and such that $K_0(X)\ne K_0(\bbL)$ for the line bundle associated to
an ample $\calL$; it is a cusp bundle over an elliptic curve.

As in our previous papers \cite{chsw, chw-v, chww-fibrant},
our technique is to compare $K$-theory to cyclic homology
using $cdh$-descent and cyclic homology.  The parts
of $cdh$ descent we need are developed in Section 1, and applied
to give a formula for the cyclic homology of line bundles in Section 2.
The main theorem is proven in Section 3, and the two examples
are given in Section \ref{sec:4}.  

\begin{notation}
If $E$ is a presheaf of spectra, we write $\pi_nE$ for the presheaf
of abelian groups $X\mapsto \pi_nE(X)$; we say that a spectrum
$E$ is {\it $n$-connected} if $\pi_qE=0$ for all $q\le n$.
For example, $K_n(X)$ is the homotopy group
$\pi_nK(X)$ of the spectrum $K(X)$.

Similarly, if $E$ is a cochain complex of
presheaves, we may regard it as a presheaf of spectra via Dold-Kan
\cite[ch.\,10]{WH}.
Thus $\pi_iE(X)$ is another notation for $H^{-i}E(X)$.
We will use the cochain shift convention $E[i]^n=E^{i+n}$,
so that the spectrum corresponding to $E[1]$ is the suspension of
the spectrum of $E$, and $\pi_nE[1]=\pi_{n-1}E$.
Thus if $E$ is $n$-connected then $E[1]$ is
$(n+1)$-connected.
\end{notation}

\bigskip
\section{Zariski and $cdh$ descent}

In this paper, we fix a field of characteristic~0, and work with
the category $\Sch$ of schemes $X$ of finite type over the field.
We will be interested in the Zariski and $cdh$ topologies on $\Sch$.

If $\tau$ is a Grothendieck topology on $\Sch$, there is an
``injective $\tau$-local'' model structure on the category
$\Psh(\Ch(\Ab))$ of
presheaves of cochain complexes of abelian groups on $\Sch$.
In this model structure, a map $A\to B$ is a cofibration
if $A(X)\to B(X)$ is an injection for all $X$, and
it is a weak equivalence if $H^n A\to H^n B$ induces an isomorphism
on the associated $\tau$-sheaves. The fibrant replacement of
$A$ in this model structure is written as $A\to\bbH_\tau(-,A)$.
We say that $A$ {\it satisfies $\tau$-descent} if the
canonical map $A(X)\to\bbH_\tau(X,A)$ is a
quasi-isomorphism for all $X$.  There is a parallel notion of
$\tau$-descent for presheaves of spectra.

If $A$ is a sheaf then $A\to\bbH_\tau(-,A)$ is an injective resolution;
it follows that $\bbH^n_\tau(X,A) = H^n\bbH_\tau(X,A)$ for all $n$.
For a complex $A$, the hypercohomology group $\bbH^n_\tau(X,A)$ equals
$H^{n}\bbH_\tau(X,A).$ See \cite[3.3]{chsw} for these facts.

The inclusion of complexes of sheaves (for a topology $\tau$)
into complexes of presheaves
induces an injective $\tau$-local model structure on complexes of
sheaves, and the inclusion is a Quillen equivalence;
see \cite[5.9]{jardine-LHT}.

For the Zariski, Nisnevich and $cdh$ topologies, there is a parallel
``injective $\tau$-local'' model structure on the category
$\Psh(\Ch(\cO_\tau))$ of presheaves of
complexes of $\cO_\tau$-modules, and the functor forgetting the module
structure is a Quillen adjunction.  In particular, if $A$ is a
presheaf of complexes of $\cO_\tau$-modules, the forgetful functor
sends its fibrant $\cO_\tau$-module replacement to a presheaf that is
objectwise weak equivalent to $\bbH_\tau(-,A)$.

\begin{ex}\label{HH descent}
The Hochschild complex $HH/k$ satisfies Zariski descent by
\cite[0.4]{WG}.  By definition, the cochain complex
$HH(X/k)$ is concentrated in negative cohomological degrees
and has the (quasi-coherent) Zariski sheaf $\cO_X^{\oo_k n+1}$ in
cohomological degree $-n$.
When $k$ is understood, we drop the '$/k$' from the notation.
We sometimes regard $HH$ as a sheaf of spectra, using Dold-Kan,
and use the notation $HH_{q}(X)=\pi_{q} HH(X)$ for
$\bbH_\zar^{-q}(X,HH)$.
Recall from  \cite[4.6]{WG} that
if $X$ is noetherian then $HH_q(X)=0$ for $q<-\dim(X)$.
\end{ex}


If $E$ is a complex of Zariski sheaves of $\cO$-modules on $\Sch/X$,
we may assume that $\bbH_\zar(-,E)$ is a complex of Zariski sheaves
of $\cO$-modules, and similarly for $\bbH_\cdh(-,E)$.
(See \cite[8.6]{jardine-LHT}.)
Thus it makes sense to form the sheaf tensor product
$\bbH_\tau(-,E)\oozar \calL$ 
with a Zariski sheaf $\calL$ of $\cO_X$-modules.

If $E$ is a Zariski sheaf of $\cO_X$-modules on $X$, then there is a Zariski
sheaf $E^\prime$ of $\cO$-modules on $\Sch/X$, unique up to unique isomorphism,
such that for every $f: Y \to X$ in $\Sch/X$ the restriction of $E^\prime$ to the small Zariski site of $Y$ is naturally isomorphic to the sheaf $f^*E$. In this paper we will always
work with this sheaf on the big site; so for example "an invertible sheaf $\calL$ on $X$" will indicate the sheaf on the big site associated in this way to an invertible sheaf on
$X$.


\begin{lem}\label{Hzar-autoeq}
If $\calL$ is an invertible sheaf on $X$, $\oozar\calL$
is an auto-equivalence of the category $\Sh(\Ch(\cO_\zar))/X$
of sheaves of complexes of $\cO_\zar$-modules on $\Sch/X$
which preserves cofibrations, fibrations and weak equivalences.
\end{lem}

\begin{proof}
The functor $\oozar\calL^{-1}$ is a quasi-inverse to $\oozar\calL$.
Since $\calL$ is flat, $\oozar\calL$ preserves injections.
Since $\calL$ is locally trivial on $X$ (and hence on any $X$-scheme), and
$A\oozar\cO_X\cong A$,
$\oozar\calL$ preserves weak equivalences. Now suppose that
$C\to D$ is a Zariski-local fibration; we want to see that
$C\oozar\calL \to D\oozar\calL$ is a Zariski-local fibration.
By invertibility, it suffices to observe that if $A\to B$ is a
trivial cofibration of $\cO_\zar$ modules, then so is
$A\oozar\calL^{-1}\to B\oozar\calL^{-1}$, a fact we have
just verified.
\end{proof}

\begin{cor}\label{Hzar-L}
If $\calL$ is an invertible sheaf on $X$, and $A$ is a complex of
Zariski sheaves of $\cO$-modules, then there is a quasi-isomorphism
on $\Sch/X$:
\[
\bbH_\zar(-,A)\oozar \calL \map{\simeq} \bbH_\zar(-,A\oozar \calL).
\]
\end{cor}

\begin{proof}
This follows immediately from Lemma \ref{Hzar-autoeq}.
\end{proof}

We write  $(a^*,a_*)$ for the usual
adjunction between Zariski and $cdh$ sheaves associated to the
change-of-topology morphism $a: (Sch/k)_\cdh\to (Sch/k)_\zar$.
Thus if $\cF$ is a sheaf of $\cO_\cdh$-modules on $(\Sch/X)_{\cdh}$,
$a_*\cF$ is the underlying sheaf of $\cO_\zar$-modules,
and for any Zariski sheaf $E$ of $\cO_X$-modules on $X$,
we may form the Zariski sheaf $a_*\cF\oo_{\cO_X}E$ on $\Sch/X$.

\smallskip
Recall from \cite[$0_I$(5.4.1)]{EGA} that a Zariski sheaf $E$ of
$\cO_X$-modules
is {\it locally free} if each point of $X$ has an open neighborhood
$U$ such that $E|_U$ is a free $\cO_U$-module, possibly of infinite rank.

\begin{lem}\label{EoF cdh}
If $E$ is a locally free sheaf on $X$, and
$\cF$ is a $cdh$ sheaf of $\cO_\cdh$-modules, then
$a_*\cF\oo_{\cO_X}E$ is a $cdh$ sheaf on $(\Sch/X)$.
\end{lem}

\begin{proof}
Since the question is local on $X$, we may replace $X$ by an open
subscheme to assume that $E$ is free.
Because the $cdh$-topology on $\Sch /X$ is quasi-compact, and
therefore arbitrary direct sums of sheaves are sheaves,
we are reduced to the trivial case $E=\cO_X$, when
$a_*\cF\oo_{\cO_X}E=a_*\cF$.
\end{proof}
%

\begin{defn}
If $\cF$ is a $cdh$ sheaf of $\cO_\cdh$-modules,
we will write $\cF\oozar E$ for the $cdh$ sheaf $a_*\cF\oo_{\cO_X}E$.
\end{defn}

Note that $\bbH_\zar^*(X,\cF\oozar E) \ne \bbH_\zar^*(X,\cF)\oo E(X)$.
For example, $E(X)=0$ does not imply that $(\cF\oozar E)(X)=0$.

\begin{lem}
If $E$ is locally free on $X$ 
then $\oozar E$
preserves weak equivalences and cofibrations 
for complexes of $cdh$ sheaves of $\cO_\cdh$-modules on $\Sch/X$.
\end{lem}

\begin{proof}
As in the proof of Lemma \ref{EoF cdh},
we may replace $X$ by an open subscheme and
assume that $E$ is a sheaf of free modules.
Since $A\oozar E$ is a sum of copies of $A$, it follows that
$A\mapsto A\oozar E$ preserves weak equivalences and cofibrations.
\end{proof}

%


\begin{defn}\label{F_E(X)}
Given a cochain complex $A$ of presheaves of abelian groups on $\Sch$, we
write $F_A(X)$ for the homotopy fiber (the shifted mapping cone)
of the canonical map $A(X)\to \bbH_\cdh(X,A)$,
so  for each $X$ there is a long exact sequence
\begin{equation*}
\cdots\bbH_\cdh^{n-1}(X,A)\to H^{n} F_A(X)\to H^{n} A(X)\to
 \bbH_\cdh^{n}(X,A) \to H^{n+1} F_A(X) \cdots.
\end{equation*}
If $A$ is a complex of sheaves (in some topology) of $\cO$-modules, then $\bbH_\cdh(X,A)$ can be represented by a complex of sheaves of $\cO$-modules as well (see
\cite[8.1]{jardine-LHT}), and hence so can $F_A$.
We also write $F_K(X)$ for the homotopy fiber of $K(X)\to KH(X)$.
\end{defn}

It is well known that $HH$, $HC$ and $K$-theory satisfy
Zariski descent; it follows that $F_{HH}$, $F_{HC}$ and $F_K$
also satisfy Zariski descent.

\begin{prop}\label{commuteH/L}
If $\calL$ is an invertible sheaf on $X$ and $A$ is a
complex of Zariski sheaves of $\cO$-modules on $\Sch/X$, then:
\[
\bbH_\cdh(-,A)\oozar\calL \map{\simeq} \bbH_\cdh(-,A\oozar\calL).
\]
Consequently, $F_A\oozar\calL \map{\simeq} F_{A\oo\calL}$.
\end{prop}

\begin{proof}
Arguing as in the proof of Lemma \ref{Hzar-autoeq},
Lemma \ref{EoF cdh} shows that $\oozar \calL$ preserves $cdh$-local fibrations
(in addition to cofibrations and weak equivalences). The first statement
follows immediately from this.
Because $\oozar \calL$ is exact, the second statement
follows from the triangles
\[
F_A\to A\to \bbH_\cdh(-,A)\to \quad\text{and}\quad
F_{A\oo \calL} \to A\oo \calL \to \bbH_\cdh(-,A\oo \calL) \to.
\qedhere\]
\end{proof}

\begin{lem}\label{lem:directsumreplace}
let $A_i$ be cochain complexes of presheaves on $\Sch/X$.
Then for every $X$-scheme $Y$, the canonical maps
\[
\bigoplus\nolimits_i \bbH_\zar(Y,A_i)\to \bbH_\zar(Y, \bigoplus\nolimits_i A_i)
\]
and
\[
\bigoplus\nolimits_i \bbH_\cdh(Y,A_i)\to \bbH_\cdh(Y, \bigoplus\nolimits_i A_i)
\]
are quasi-isomorphisms.
\end{lem}

\begin{proof}
These sites are quasi-compact, and thus
cohomology in them commutes with direct limits.
\end{proof}

\smallskip
\section{Homology of line bundles}\label{sec:HH(R[L])}

Suppose that $R$ is a (commutative) noetherian algebra
over a field $k$ of characteristic~0.
In \cite[3.2, 4.1]{chww-fibrant}, we showed that $NK(R)=K(R[t])/K(R)$
is isomorphic to $NF_{HC/\Q}(R)[1]$ as well as
$\F_{HH/\Q}(R)[1]\oo_R tR[t]$. In this section, we
replace $R[t]$ by the symmetric algebra $R[L]=\Sym_R(L)$ of a
rank 1 projective $R$-module, and the ideal $tR[t]$ by $LR[L]$.
More generally, if $\calL$ is an
invertible sheaf on a scheme $X$, we replace $X\times\A^1$ by
the line bundle $\bbL=\Spec(\Sym_X \calL)$.

\begin{lem}\label{aff:Kunneth}
Let $L$ be a rank 1 projective $R$-module. Then the symmetric algebra
$R[L]=\Sym_R(L)$ satisfies:
\begin{align*}
HH(R[L]) \simeq&\; HH(R)\oo_RR[L]  \;\oplus\; HH(R)[1] \oo_R LR[L] \\
HC(R[L]) \simeq&\; HC(R) \;\oplus\; HH(R)\oo_R LR[L]. 
\end{align*}
Sinilarly, if $X$ is a scheme over $R$ and $X[L]$ denotes
$X\times_R\Spec(R[L])$, then
\begin{align*}
HH(X[L]) \simeq&\; HH(X)\oo_RR[L]  \;\oplus\; HH(X)[1]\oo_R LR[L]) \\
HC(X[L]) \simeq&\; HC(X) \;\oplus\; HH(X)\oo_R LR[L].  
\end{align*}
\end{lem}

Note that, as an $R$-module, $LR[L] = R[L]\oo_R L$ is just
$\oplus_{j=1}^\infty L^{\oo j}$.

\begin{proof}
The cochain complex $HH(R[L])$ ends:
$\to R[L]\oo R[L] \smap{0} R[L]\to0$.
Therefore there are natural maps from $R[L]$ and $R[L]\oo L[1]$ to $HH(R[L])$.
Using the shuffle product, we get a natural map $\mu(R)$ from the direct sum of
$HH(R)\oo_R R[L]$ and $HH(R)\oo_R(R[L]\oo L)[1]$ to $HH(R[L])$.
For each prime ideal $\wp$ of $R$, we have $R_\wp[L]\cong R_\wp[t]$
and $\mu(R_\wp)$ is a quasi-isomorphism by the K\"unneth formula
\cite[9.4.1]{WH}. It follows that $\mu(R)$ is a quasi-isomorphism.
The formula for $HC(R[L])$ follows by induction on the SBI sequence,
just as it does for $HC(R[t])$.

If $X$ is a scheme over $R$, the same argument applies to
$\pi_* HH(\cO_X [L])$, the direct image along 
$X[L]\smap{\pi} X$
of the cochain complex $HH(\cO_X [L])$ on $X[L]$
of quasi-coherent sheaves described in Example \ref{HH descent}.
Because $\pi$ is affine, we have a quasi-isomorphism
$$\bbH_\zar(X[L],HH(\cO_X [L]))\cong \bbH_\zar(X,\pi_* HH(\cO_X [L]).$$
Now the assertions about $X[L]$ follow from
Corollary \ref{Hzar-L} and Lemma \ref{lem:directsumreplace}.
\end{proof}



\begin{cor}\label{cor:Kunneth}
$F_{HC}(R[L]) \cong F_{HC}(R) \;\oplus\;
\bigoplus\nolimits_{j=1}^\infty \left(F_{HH}\oo_R L^{\oo j}\right)(R).$
\end{cor}

\begin{proof}
This follows from Lemma \ref{aff:Kunneth}, Proposition \ref{commuteH/L},
and Lemma \ref{lem:directsumreplace}.
\end{proof}



Now suppose that $X$ is a scheme of finite type over a field
of characteristic~0, containing $k$,
and write $HH$, $HC$, etc for $HH/k$, $HC/k$, etc.

\begin{lem}\label{descent}
Let $\bbL$ be a line bundle over $X$, and write $\cF_{HH}$ for the cochain
complex of Zariski sheaves on $X$ associated to the complex of presheaves
$U\mapsto F_{HH}(\bbL|_U)$.
Then $F_{HH}(\bbL) \smap{\simeq} \bbH_\zar(X,\cF_{HH})$.
\end{lem}

\begin{proof}
As observed after \ref{F_E(X)}, the presheaf of complexes $F_{HH}$
satisfies Zariski descent: $F_{HH}(\bbL)\!\simeq\!\bbH_\zar(\bbL,F_{HH})$.
By \cite[1.56]{AKTEC},  $\bbH_\zar(\bbL,F_{HH})\smap{\simeq}
\bbH_\zar(X,\cF_{HH})$.
\end{proof}

In what follows, we write $\oo$ for the tensor product of $\cO_X$-modules.

\begin{prop}\label{Kunneth}
Let $\bbL$ be the line bundle $\Spec(\Sym\,\calL)$ on $X$ associated to
an invertible sheaf $\calL$, and $p:\bbL\to X$ the projection.
Then we have quasi-isomorphisms:
\begin{align*}
HC(\bbL) \simeq\ & HC(X) \oplus \bbH_\zar(X,HH\oo\Sym(\calL)\oo\calL);
\\
\bbH_\cdh(X,p_*HC) \simeq\ &
\bbH_\cdh(X,HC)\oplus \bbH_\cdh(X,HH\oo\Sym(\calL)\oo\calL);
\\
F_{HC}(\bbL) \simeq\ &
F_{HC}(X)  \;\oplus\;
\bigoplus\nolimits_{j=1}^\infty\left(F_{HH}\oo\calL^{\oo j}\right)(X);
\\
K(\bbL,X) \simeq\ & F_{HC}(\bbL,X)[1].
\end{align*}
\end{prop}

\begin{proof}
 Using Zariski descent, we may assume that $X=\Spec(R)$ for some $R$.
The first two quasi-isomorphisms are immediate from
Lemma \ref{aff:Kunneth}, while the third is immediate from
Corollary \ref{cor:Kunneth}.
By Theorem 1.6 of \cite{chw-v},
\[
K(\bbL)/K(X)\cong F_K(\bbL)/F_K(X) \simeq
F_{HC/\Q}(\bbL)[1]/F_{HC/\Q}(X)[1].
\]
Now use
the formula for $F_{HC}(\bbL)$
to get the final quasi-isomorphism.
\end{proof}

Now suppose that $R$ is a commutative $\Q$-algebra.  Then
$K_n(R[L],R)$ is a $\Q$-module \cite{W-mod}, and
the Adams operations give an $R$-module decomposition
\[
K_n(R[L],R)\cong\oplus_{i=0}^\infty K_n^{(i)}(R[L],R)
\]
with $K_n^{(0)}(R[L],R)=0$ for all $n$. The relative terms
$F_K(R)\cong F_{HC}(R)[1]$ have a similar decomposition, and
$F_K^{(i)}(R[L],R) \simeq F_{HC}^{(i-1)}(R[L],R)[1]$.

As in \cite[5.1]{chww-fibrant}, we define the {\it typical piece}
$TK_n(R)$ to be $H^{1-n}(F_{HH}(R))$, and set
$TK_n^{(i)}(R)=H^{1-n}(F_{HH}^{(i-1)}(R))$.
Since these groups were detemined in \cite{chww-fibrant},
we may rephrase the last part of Proposition \ref{Kunneth} as follows:

\begin{cor}\label{TK(RL)}
If $R$ is a commutative $\Q$-algebra,
$K_n(R[L],R)\cong TK_n(R)\oo_R LR[L]$ and
\[
K_n^{(i)}(R[L]) \cong K_n^{(i)}(R) \oplus TK_n^{(i)}(R)\oo_R LR[L].
\]
Moreover,
\[
TK_n^{(i)}(R) \cong\begin{cases}
HH_{n-1}^{(i-1)}(R), &\text{if }~ i<n,\\
H_\cdh^{i-n-1}(R,\Omega^{i-1}),& \text{if }~ i\ge n+2.
\end{cases}\]
\end{cor}

(The formulas for $TK_n^{(n)}$ and $TK_n^{(n+1)}$ are more complicated;
see {\it loc.\,cit}.)
The following special case $n=0$ of \ref{TK(RL)}, which is
an analogue of \cite[(0.5)]{chww-fibrant},
shows that we cannot twist out the example in \cite[Theorem 0.2]{chw-v}

\begin{cor}
Let $L$ be a rank 1 projective $R$-module, where $R$ is a $d$-dimensional
commutative $\Q$-algebra, with seminormalization $R^+$,
and $R[L]$ the twisted polynomial ring.  Then
\[
K_0(R[L],R)\cong \biggl((R^+/R)\oplus
\bigoplus_{p=1}^{d-1}\bbH_\cdh^{p}(R,\Omega^{p}) \biggr)\oo_R LR[L].
\]
In particular, $K_n(R)=K_n(R[t])$ if and only if $K_n(R)=K_n(R[L])$.
\end{cor}

\begin{proof}
This follows from the fact that
$\bbH_\cdh(X,HH^{(i)})\cong Ra_*a^*\Omega^i[-i]$, so that
when $i>1$ we have
$K_0^{(i)}(R[L],R)\cong\bbH_\cdh^{i-1}(R,\Omega^{i-1})\oo_R LR[L]$;
see \cite[2.2]{chw-v}.
\end{proof}

\begin{rem}\label{S-module}
Corollary \ref{TK(RL)} shows that $K_*(R[L],R)$ is a graded
$R[L]$-module. As in \cite{chww-fibrant},
this reflects the fact that locally $R[L]$ is a polynomial
ring, and $K_*(R[t],R)$ has a continuous module structure over
the ring of big Witt vectors $W(R)$, compatible with the operations
$V_n$ and $F_n$; when $\Q\subset R$,
such modules are graded $R[t]$-modules.
Since $H^0(\Spec R,\widetilde W)=W(R)$, patching
the structures via Zariski descent proves that $K_*(R[L],R)$
is a graded $R[L]$-module.

When $X$ is no longer affine, this Zariski descent argument
shows that
\[
K_n(\bbL,X) = \oplus H^{1-n}(X,F_{HH}\oozar \calL^{\oo i})
\]
is a graded module over $S=\oplus H^0(X,\calL^{\oo i})$.
This is clear from Proposition \ref{Kunneth}.
Previously, using \cite{W-mod}, it was only known that the
$K_n(\bbL,X)$ are continuous modules over
$H^0(X,\widetilde{W})=W(k)=\prod_1^\infty k$.
\end{rem}

\medskip
\section{Proof of Theorem \ref{main}}

In order to use Proposition \ref{Kunneth},
we need to analyze
$\bbH_\zar^n\left(X,F_{HH/\Q}\oo\calL^{j}\right)$.
For this, we use the hypercohomology spectral sequence.
(See \cite[5.7.10]{WH}.)
\begin{equation}\label{eq:SS}
E_2^{p,q} = H^p_\zar(X,H^qE) \Rightarrow \bbH_\zar^{p+q}(X,E).
\end{equation}
Here $E$ is a cochain complex which need not be bounded below and
(by abuse of notation) the $E_2^{}$ term denotes cohomology with
coefficients in the Zariski sheaf associated to $H^qE$;
the spectral sequence converges if $X$ is noetherian and finite dimensional.
When $E=F_{HH}\oo\calL^j$, we have $H^qE=H^q(F_{HH})\oo\calL^j$,
because $\calL^j$ is flat.

\begin{lem}\label{l:hyper}
If $X$ is noetherian and finite dimensional, and
$E$ is a complex of Zariski sheaves
such that
$H_\zar^p(X,H^qE)=0$ for $1\le p\le\dim(X)$ and $p+q=s,s+1$
then $\bbH_\zar^s(X,E)\cong H_\zar^0(X,H^sE)$.
\end{lem}

\begin{proof}
This is immediate from the hypercohomology spectral
sequence \eqref{eq:SS}.
\end{proof}

In the remainder of this section,  we will write $H^p(X,-)$ for $H^p_\zar(X,-)$.
By a ``quasi-coherent'' (or ``coherent'') sheaf on $\Sch/k$ we mean a
Zariski sheaf whose restriction to every small Zariski site is
quasi-coherent (or coherent).
When discussing Hochschild homology (or
cyclic homology, or differentials, etc.\,) relative to $\Q$, we will
suppress the base from the notation. For example, if $X$ is a $k$-scheme
then $HH_n(X)$ and $\Omega^n_X$ will mean $HH_n(X/\Q)$ and $\Omega^n_{X/\Q}$.

Recall that when $\Q\subseteq k$, the Hochschild homology complex
relative to $k$ decomposes into a direct sum of weight pieces
$HH^{(j)}(-/k)$; this induces decompositions on $\bbH_\cdh (-,
HH(/k))$, the fiber $F_{HH(/k)}$, and on their cohomology sheaves and
hypercohomology groups as well. As in \cite{chw-v}, we use versions of
a spectral sequence introduced by Kassel and Sletsj\oe\ in \cite{kasle}
to obtain information about $F_{HH(/k)}$ from information about
$F_{HH}$.

\goodbreak
\begin{lem}\label{lem:KS} (Kassel-Sletsj\oe)
Let $\Q\subseteq k$ and $p\ge1$ be fixed, and $X$ a scheme over $k$.
Then there are bounded cohomological spectral sequences
of quasi-coherent sheaves on $\Sch/k$ ($p>s\ge0$):
\[
E_1^{s,t} = \Omega^s_k \oo_k H^{2s+t-p}HH^{(p-s)}(-/k)
\Rightarrow H^{s+t-p}HH^{(p)}(-/\Q)
\]
(for $s+t\le0$) and
\[
E_1^{s,t} = \Omega^s_k \oo_k H^{s+t}(Ra_*\Omega_{(-/k),\cdh}^{(p-s)})
\Rightarrow H^{s+t}(Ra_*\Omega_\cdh^p)
\]
and a morphism of spectral sequences between them. If $k$ has finite
transcendence degree, then both spectral sequences are spectral
sequences of coherent sheaves.
\end{lem}

We remark that the second spectral sequence is just the sheafification
of the spectral sequence in \cite[4.2]{chw-v}.

\begin{proof}
If $X=\Spec(R)$, the homological spectral sequence in \cite[4.3a]{kasle} is
\[
{}_p E^1_{-i,i+j}=\Omega^{i}_{k}\oo_k HH^{(p-i)}_{p-i+j}(R/k)
\Rightarrow HH_{p+j}^{(p)}(R)
\]
($0\le i<p$, $j\ge0$); see \cite[4.1]{chw-v}.

We claim that this is a spectral sequence of $R$-modules, compatible
with localization of $R$. Indeed, following the construction in
\cite[Theorem 3.2]{kasle}, the exact couple underlying the spectral
sequence is constructed by choosing $\Q$-cofibrant simplicial
resolutions $P_\mathdot \to k$ and $Q_\mathdot \to R$ and then filtering
the differential modules $\Omega^p_{Q_\mathdot/\Q}$ by certain
$Q_\mathdot$-submodules, leading to a filtration of
$\Omega^p_{Q_\mathdot/\Q}\otimes_{Q_\mathdot} B$ by $B$-modules.
(Although the filtration steps are defined as certain
$P_\mathdot$-submodules in \cite[Section 3]{kasle}, they are in
fact $Q_\mathdot$-submodules.)
The identification of the associated graded via
\cite[Lemma 3.1]{kasle} is easily checked to be a $B$-module
isomorphism. The whole construction commutes with localization because
forming differential modules does.

Setting  $\ell=i+j$, the spectral sequence is
\[
{}_p E^1_{-i,\ell}=\Omega^{i}_{k}\oo_k HH^{(p-i)}_{p+\ell-2i}(R/k)
\Rightarrow HH_{p+\ell-i}^{(p)}(R),\quad \ell\le i.
\]
As this spectral sequence is a spectral sequence of $R$-modules,
compatible with localization and  natural in $R$, we may sheafify it for
the Zariski topology to obtain a spectral sequence of quasi-coherent sheaves.
Reindexing cohomologically, with $s=i$ and $t=-\ell$, we have
\[
{}^p E_1^{s,t}=\Omega^{s}_{k}\otimes_k H^{2s+t-p}(HH^{(p-s)})(-/k)
\Rightarrow H^{s+t-p}(HH^{(p)}).
\]
This yields the first spectral sequence.
If we sheafify it for the $cdh$ topology, and use the isomorphism
$HH^{(p)}_\cdh \cong \Omega^p_\cdh[p]$, we get the second spectral sequence.
That it is still a spectral sequence of quasi-coherent sheaves follows
from \cite[lemma 2.8]{chw-v}.
The morphism between the spectral sequences is just the
change-of-topology map.

Finally, if $k$ has finite transcendence degree, then the $E_1$-terms
of both spectral sequences are coherent (apply \cite[lemma 2.8]{chw-v}
again for the second one) and hence so are the abutments.
\end{proof}

\begin{cor}\label{cor:KS}
There is a bounded spectral sequence of quasi-coherent sheaves
\[
E_1^{s,t} = \Omega^s_k \oo_k H^{2s+t-p}(F_{HH/k}^{(p-s)})
\Rightarrow H^{s+t-p}(F_{HH}^{(p)}).
\]
If $k$ has finite transcendence degree, this is a spectral
sequence of coherent sheaves.
\end{cor}

\begin{proof}
The morphism of spectral sequences in Lemma \ref{lem:KS}
comes from a morphism $HH^{(p)}\to HH^{(p)}_\cdh$
of filtered complexes of quasi-coherent sheaves on $\Sch/k$.
By a lemma of Eilenberg--Moore \cite[Ex.\,5.4.4]{WH}, there is a
filtration on the [shifted] mapping cone $F_{HH}^{(p)}$
of $HH^{(p)}\to HH_\cdh^{(p)}$, yielding a
spectral sequence converging to $H^*(F_{HH})$.
This is the displayed spectral sequence.
\end{proof}

\begin{prop}\label{twist}
Assume that $k$ has finite transcendence degree. If $\calL$ is
an ample line bundle on $X$, then for every
$n$ and $p\geq 0$ there is an $N_0 = N_0(n,p)$ such that for all $N>N_0$
the Zariski sheaf $H^nF_{HH}^{(p)}\oo\calL^{\oo N}$
is generated by its global sections,
and $H^q(X,H^nF_{HH}^{(p)}\oo\calL^{\oo N})=0$ for all $q>0$.
\end{prop}

\begin{proof}
  The complex $F_{HH}^{(0)}$ is quasi-isomorphic to the cone of the
  map from the structure sheaf $\cO$ to $Ra_* a^* \cO$ and thus has
  coherent cohomology by \cite[Lemma 6.5]{chsw}. If $p > 0$, then by
  Corollary \ref{cor:KS} the cohomology sheaves in question are
  coherent as well. Now apply Serre's Theorem B.
\end{proof}

Let $\calL$ be an ample sheaf on $X$ and $\bbL$ the line bundle
$\Spec(\Sym\,\calL)$.  Recall that for any $Y$, $F_{HC}(Y)$ is
$n$-connected if and only if $F_{HH}(Y)$ is $n$-connected; see
\cite[1.7]{chw-v}. If $\bbL$ is a line bundle over $X$, we define
$\F_{HH/k}(\bbL,X)$ to be the
cokernel of the canonical split injection
$\F_{HH/k}(X)\to\F_{HH/k}(\bbL)$, and similarly for cyclic homology.

\begin{thm}\label{thm:HH}
If $F_{HC}(\bbL,X)$ is $n$-connected for some ample line bundle
$\calL$ on $X$, then $F_{HH}(\bbL,X)$ is $n$-connected and:
\begin{enumerate}
\item  The Zariski sheaf $F_{HH}$ is $n$-connected.
\item  $X$ is regular in codimension $\le n$.
\item If $F_{HC}(\bbL,X)$ is $d$-connected for $d=\dim(X)$,
then $X$ is regular.
\end{enumerate}
\end{thm}

\begin{proof}
There is a finitely generated subfield $k_0$ of $k$,
a $k_0$-scheme $X_0$ and an ample line bundle $\calL_0$ such that $X =
X_0\otimes_{k_0} k$ and $\calL = \calL_0\otimes_{k_0} k$. The
K\"unneth formula for Hochschild homology implies that $F_{HH}(\bbL,X)
= F_{HH}(\bbL_0,X_0)\otimes\Omega^\ast_{k/k_0}$, whence
$F_{HH}(\bbL,X)$ is $n$-connected if and only if $F_{HH}(\bbL_0,X_0)$
is. Thus we may assume that $k$ has finite transcendence degree.

(1)
Recall \cite[2.1]{chw-v} that $F_{HH}(\bbL,X)=\prod F_{HH}^{(p)}(\bbL,X)$.
Thus it suffices to fix $p$ and show that $F_{HH}^{(p)}$ is $n$-connected.
Set $\cG_N=\calL^N\oo F^{(p)}_{HH}$,
and note that $H^q\cG_N=\calL^N\oo H^qF^{(p)}_{HH}$.
By Proposition \ref{twist} and
Lemma \ref{l:hyper}, $H^s(X,\cG_N)\cong H^0(X,H^s\cG_N)$ for large $N$
and all $s\ge-n$.

By assumption and Lemma \ref{descent}, the groups
\[
\pi_sF_{HH}^{(p)}(\bbL,X)= \bbH_\zar^{-s}(X,F^{(p)}_{HH}(\bbL,X))
=\bbH_\zar^{-s}(X,F^{(p)}_{HH}(\bbL)/F^{(p)}_{HH})
\]
vanish for $s\le n$.
By Lemma \ref{Kunneth}, this implies that for all $N>0$:
\[
H^0(X,H^{-s}\cG_N) \cong H^{-s}(X,\cG_N) =
H^{-s}(X,\calL^N\oo F^{(p)}_{HH})=0,
s\le d.
\]
Since $\calL$ is ample, the sheaves $H^s\cG_N=\calL^N\oo H^sF^{(p)}_{HH}$
are generated by their global sections $H^0(X,H^s\cG_N)$ for large $N$
and $s\ge-n$. This implies that the sheaves $\calL^N\oo H^sF^{(p)}_{HH}$
vanish, and hence that the sheaves $H^sF^{(p)}_{HH}$ vanish for $s\ge-n$.
This proves (1).

Given (1), the stalks $F_{HH}(\cO_{X,x})$ are $n$-connected.
We proved in \cite[4.8]{chw-v} that this implies that each
$F_{HH/k}(\cO_{X,x})$ is $n$-connected. If $\dim(\cO_{X,x})\ge n$, we proved in
\cite[3.1]{chw-v} that $\cO_{X,x}$ is smooth over $k$, and hence regular.
\end{proof}

\begin{variant}\label{thm:May16}
Let $X$, $\calL$ and $\bbL$ be as in Proposition \ref{thm:HH}.
Suppose that $F_{HC/k}(\bbL,X)$ is $n$-connected.
Then the proof of Theorem \ref{thm:HH} goes through to show that:
\begin{enumerate}
\item The sheaf $\cF_{HH/k}$ is $n$-connected.
\item  $X$ is regular in codimension $\le n$.
\item If $F_{HH/k}(\bbL,X)$ is $d$-connected for $d=\dim(X)$,
then $X$ is regular.
\end{enumerate}
\end{variant}

\begin{proof}[Proof of Theorem \ref{main}]
Suppose that $K_i(\bbL)\cong K_i(X)$ for all $i\le n$.
By Proposition \ref{Kunneth}, $F_{HC/\Q}(\bbL,X)$
is $(n-1)$-connected.
By Theorem \ref{thm:HH}, $F_{HH/\Q}(\bbL,X)$ is $(n-1)$-connected and
$X$ is regular in codimension $<n$
\end{proof}

\section{Two examples}\label{sec:4}

We conclude with two quick examples.
Let $E$ be an elliptic curve over $\Q$ with basepoint $Q$, and
$P$ a point such that $P-Q$ does not have finite order in $\Pic(E)$.

\begin{ex}\label{ex:DD}
Consider the non-reduced scheme $Y=\Spec(\cO_E\oplus J)$, where
$J$ is the invertible sheaf $\cO(P-Q)$.  We showed in
\cite[0.2]{chw-v} that $Y$ is $K_n$-regular for all $n$, because
$K_n(Y\times\A^1)\cong K_n(Y)\cong K_n(E)$ for all $n$.

Let $\calL$ be the sheaf $\cO(Q)$ and set $\bbL=\Spec_Y(\Sym\calL)$.
Then
\[
K_0(\bbL) \cong K_0(Y) \oplus \Q[x,y].
\]
\end{ex}

\smallskip
For our second example, recall that if $R$ is a regular $\Q$-algebra
and $J$ is a rank~1 projective $R$-module and $A$ is the subring
$R[J^2,J^3]$ of $R[J]=\Sym_R(J)$ then $\Spec(A)$ is an affine cusp
bundle over $\Spec(R)$.  For $n\ge2$, set
\[
V_n(R) =
\begin{cases}
J^{6(i-1)}\oplus (J^{6(i-2)}\oo\Omega_R^2)\oplus\cdots
 \oplus (R\oo\Omega_R^{n-2}), & n=2i\ge2; \\
J^{6(i-1)}\oo\Omega_R^1\oplus(J^{6(i-2)}\oo\Omega_R^3)\oplus\cdots
 \oplus (R\oo\Omega_R^{n-2}), & n=2i+1\ge3.
\end{cases}
\]
In particular, $V_2(R)=R$ and $V_3(R)=\Omega_R^1$.
Let us write $\widetilde{K}_n(A)$ for $K_n(A)/K_n(R)$.

\begin{prop}\label{GRW 9.2}
If $A=R[J^2,J^3]$ and $R$ is a regular $\Q$-algebra then
\[
\widetilde{K}_n(A)
\cong
(J^5\oplus J^6) \oo V_n(R) \oplus (J\oo\Omega_R^n).
\]
In particular, $\widetilde{K}_0(A)\cong J$,
$\widetilde{K}_1(A)\cong J\oo\Omega^1_R$ and
\[
\widetilde{K}_2(A)
\cong (J^5\oplus J^6) \oplus (J\oo\Omega_R^2).\]
\end{prop}

\begin{proof}
For $J=R$, this is Theorem 9.2 of \cite{GRW}, which holds for any
regular $\Q$-algebra $R$ (not just for any field). In order to
pass to $R[J^2,J^3]$, we need more detail.  Using the classical
Mayer-Vietoris sequence for $A\subset R[J]$, it is easy to see that
$K_0(A)/K_0(R)\cong J$ and $K_1(A)/K_1(R)\cong J\oo\Omega^1_R$.

For $n\ge2$ the factors in $K_n(A)$ come from $HH_{n-1}(A)$ via the maps
$HH_*(A)\to HC_*(A)$ and $\widetilde{K}_n(A)\to \widetilde{HC}_{n-1}(A)$.
The summand $J\oo\Omega_R^n$ of $K_n(A)$ comes from
the $J\oo\Omega_R^1$ in $K_1(A)$ (or $HH_0(A,R[J],J)$)
by multiplication by $HH_{n-1}(R)\cong\Omega^{n-1}_R$.

The $V_n$ factors come from the explicit description of the corresponding
cyclic homology cycles (coming from cycles in Hochschild homology
$HH_{n-1}(A)$) in 4.3, 4.7 and 5.8 of \cite{GRW}.  Locally,
$J$ is generated by an element $t$; we set $x=t^2\in J^2$,
$y=t^3\in J^3$ so that $y^2=x^3$.  The summands $J^5$ and $J^6$ of $K_2(A)$
are locally generated by the cycles $z=2x[y]+3y[x]$
and $tz=2y[y]+3x^2[x]$ in $HH_1(A)$.
Multiplication by $\Omega_R^{n-2}$ gives the summands
$(J^5\oplus J^6)\oo\Omega_R^{n-2}$ in $K_{n}(A)$.

Now consider the summand $J^6$ in the degree~2 part $A^{\oo3}$ of the
Hochschild complex for $A$, locally generated by the
element $w=[y|y]-x[x|x]-[x^2|x]$.  The product $zw^{i-1}$
is a cycle in $HH_{2i-1}(A)$, and locally generates
a summand $J^{5+6(i-1)}$ of $HH_{2i-1}(A)$, corresponding to the
factor $J^{5+6(i-1)}$ of the summand $J^5\oo V_{2i}(R)$
of $K_{2i}(A)$.  As above, multiplication by $\Omega_R^{*}$
gives the rest of the summands.
\end{proof}

\begin{subrem}\label{MVcusp}
In the spirit of Corollary \ref{TK(RL)}, we note that
$NK_n(A) \cong TK_n(A)\oo_R LR[L]$, where
\[
TK_n(A) = \widetilde{K}_n(A) \oplus \widetilde{K}_n(A).
\]
\end{subrem}

\begin{thm}\label{thm:cusp-bundle}
Let $J$ be the invertible sheaf $\cO(P-Q)$ on the ellipic curve $E$
and let $X$ denote the affine cusp bundle $\Spec_E(\cO_E[J^2,J^3])$ over $E$.
($X$ has a codimension~1 singular locus.)
If $J$ does not have finite order in $\Pic(E)$ then
$X$ is $K_n$-regular for all integers $n$: for all $m\ge0$ we have
\[  K_n(X) \cong K_n(X\times\A^m) \cong K_n(E) \]
On the other hand, if
$\bbL=\Sym_E(\cO(Q))$ then $K_{-1}(\bbL)\ne K_{-1}(X)$ and
$K_0(\bbL)\ne K_0(X)$.
\end{thm}

\begin{proof}
Since $\Omega_E\cong\cO_E$, $V_n(\cO_E)$ is a sum of terms $J^i$ for
$i>0$; the same is true for the pushforward of the sheaf
$V_n(\cO_E[t_1,...,t_m])$ to $E$.
Recall that $H^p(E,J^r)=0$ for all $r\ne0$.
From the Zariski descent spectral sequence
\[
E_2^{p,q} \!=\! H^p(E,K_{-q}(\cO_E[J^2,J^3][t_1,...,t_m])/K_{-q}(\cO_E))
\Rightarrow K_{-p-q}(X\times\A^m)/K_{-p-q}(E)
\]
we see that $K_n(X\times\A^m)\cong K_n(E)$ for all $n$.

On the other hand, Proposition \ref{GRW 9.2} yields
$\widetilde{K}_{-1}(\bbL)\cong\oplus_{j\ge1}H^1(E,J\oo\calL^j)$
and
\[
\widetilde{K}_{0}(\bbL)\cong \oplus_{j\ge1}H^0(E,J\oo\calL^j)
\oplus \widetilde{K}_{-1}(\bbL).
\]
These groups are nonzero because $\calL$ is ample.
\end{proof}

\medskip
\subsection*{Acknowledgements}
The authors would like to thank the Tata Institute, the University
of Melbourne  and the Hausdorff Institute for Mathematics
for their hospitality during the preparation of this paper.

\goodbreak

\end{document}